\documentclass[12pt]{article}

\usepackage{latexsym,amssymb,amsmath}
\newcommand{\C}{\mathbb C}
\newcommand{\Z}{\mathbb Z}

\newcommand{\R}{\mathbb R}
\newcommand{\Q}{\mathbb Q}
\newtheorem{lem}{Lemma}[section]
\newtheorem{defn}[lem]{Definition}

\newtheorem{co}[lem]{Corollary}
\newtheorem{thm}[lem]{Theorem}
\newtheorem{prop}[lem]{Proposition}

\newenvironment{proof}{\textbf{Proof.}}{\newline\hspace*{\fill}{$\Box$}\\}

\begin{document}
\title{Explicit Helfgott type growth in free products and in limit groups}
\author{J.\,O.\,Button\\
Selwyn College\\
University of Cambridge\\
Cambridge CB3 9DQ\\
U.K.\\
\texttt{jb128@dpmms.cam.ac.uk}}
\date{}
\maketitle
\begin{abstract}
We adapt Safin's result on powers of sets in free groups to obtain
Helfgott type growth in free products:
if $A$ is any finite subset of a free product of two arbitrary groups
then either $A$ is conjugate
into one of the factors, or the triple product $A^3$ of $A$
satisfies $|A^3|\ge (1/7776) |A|^2$, or $A$ generates an infinite
cyclic or infinite dihedral group. We also point out that if $A$ is any
finite subset of a limit group then $|A^3|$ satisfies the above inequality
unless $A$ generates a free abelian group.
This gives rise to many infinite groups $G$ where there exist $c>0$ and
$\delta=1$ such that any finite subset $A$ of $G$ satisfies either
$|A^3|\ge c|A|^{1+\delta}$ or it generates a virtually nilpotent group.
\end{abstract}

\section{Introduction}

When $A$ is a finite subset (always assumed non empty in this paper) of an
abelian group then there has long been interest in classifying when $A$ has
small doubling. Namely on being given a class of abelian groups (say all
torsion free abelian groups or all finite abelian groups) and a real
number $K\ge 1$, one wants to describe the structure of those $A$ such that
the size of $A+A$ is at most $K|A|$. However for finite subsets of non
abelian groups there were until recently very few equivalent results (say only
Freiman's 3/2 Theorem from 1973 which gives a complete characterisation
in any group $G$ of all finite subsets $A$ whose double product $A^2$
satisfies $|A^2|<(3/2)|A|$). An issue here is that in
an infinite non abelian group $G$ one can have (in fact
often has, see Section 4) 
a sequence of finite subsets $A$ of $G$ with $|A^2|<4|A|$ but
$|A^3|/|A|\rightarrow\infty$ as $|A|\rightarrow\infty$, so that 
control over the
size of the double product $A^2$ does not give control of the triple product
$|A^3|$ unlike in the abelian case.

A breakthrough came with the paper \cite{hlf} of Helfgott, which showed that
once $A^3$ has bounded size then so do all the higher product sets
$A^4,A^5,$ and so on. He then considered the collection of finite groups
$SL(2,p)$ over all primes $p$ and showed a strong growth property of triple
products in this case. Of course in general we cannot expect triple product 
growth for all finite subsets of a group: examples of $A$ such that
$|A^3|<2|A|$ would be when $A$ is itself a finite subgroup (or consists of
more than half of a finite subgroup), or when $A$ is an arithmetic
progression, namely it consists of
consecutive powers $\{x^n,x^{n+1},\ldots ,x^{n+i}\}$ of
any element $x$. Helfgott's result states (using the Gowers trick as in
\cite{soda})
that there are absolute constants $c,\delta >0$ such that if $A$ is any 
generating set of $G=SL(2,p)$ for any prime $p$ then either 
$|A|\ge |G|^{9/10}$ in
which case $A^3=G$, or $|A^3|\ge c|A|^{1+\delta}$ when $|A|<|G|^{9/10}$.

There has also been interest recently in approximate groups, following
\cite{tao}. Given $K\ge 1$, we say that a finite subset $A$ of a group $G$
is a $K$-approximate group if it is symmetric, contains the identity,
and there is a subset $X$ of $G$ with $|X|$ at most $K$ such that 
$A^2\subseteq XA$. An immediate consequence of this definition is that
$|A^2|\le K|A|,|A^3|\le K^2|A|$ and $|A^{n+1}|\le K^n|A|$, so certainly
the size of all product sets is well behaved here. There are results on
the structure of $K$-approximate groups, cumulating in \cite{bgst} which
completely characterises such subsets. However our interest here is in their
existence. Certainly any finite symmetric subset $A$ containing the identity
is a $K$-approximate group for some $K$, namely $K=|A|$. Thus in practice
we fix $K$ and ask how big subsets $A$ of an infinite group $G$ can be if
they are $K$-approximate groups, or more generally if they satisfy
$|A^3|\le K^2|A|$. Moreover as there are certainly examples using finite
subgroups or arithmetic progressions, we should insist that the subgroup
$\langle A\rangle$ generated by $A$ is not a virtually cyclic group. As our
interest here is in that part of the spectrum of groups close
to free groups, we will further insist that $\langle A\rangle$ is not
virtually abelian or virtually nilpotent; indeed balls in virtually
nilpotent groups also provide examples of approximate groups.

At this point it seems that a severe dearth of examples presents itself. 
For instance the
author was unable to find in the literature a single case of an infinite
group $G$ and a real number $K>1$ where there exist finite subsets $A$ of $G$
all generating non virtually nilpotent subgroups and with $|A^3|\le K^2|A|$
but $|A|$ unbounded. If there were no such examples then there would exist
a function $f(K)$, possibly depending on the ambient group $G$, such that
any $K$-approximate group $A$ in $G$ with $\langle A\rangle$ not virtually
nilpotent has $|A|\le f(K)$. In particular any theorem
describing $K$-approximate groups of infinite groups along the
lines of:
``there is a subset $S$ of $A$ (or a small power of $A$) with $S$
of size at least 
$|A|/O_K(1)$ such that $S$ is well behaved'' would hold just by taking 
$S=\{id\}$.
In fact examples do exist, with many aware of the case when $G$ is equal
to the direct product of the rank 2 free group $F_2=\langle x,y\rangle$ 
and the integers $\Z=\langle z\rangle$, whereupon we can take $A=\{(x,z^i),
(y,z^i):0\le i\le N-1\}$ with $|A^3|<8|A|$ and $\langle A\rangle=F_2\times\Z$
but $|A|=2N$. We elaborate on this example in Proposition 4.2. 

Moreover when we do have an infinite group $G$ with $f(K)$ such that 
$|A^3|\le K^2|A|$ implies that $|A|\le f(K)$ when $\langle A\rangle$ is not
virtually nilpotent, we can still ask how quickly $f$ grows. We can make a
connection here with Helfgott's results by using the following definition:
\begin{defn}
Given an infinite group $G$, we say that $G$ has 
{\bf Helfgott type growth}
if there exist $c,\delta>0$ such that 
any subset $A$ with $\langle A\rangle$
not virtually nilpotent satisfies $|A^3|\ge c|A|^{1+\delta}$.
\end{defn}
It is clear that if this definition is satisfied then
any $K$-approximate group $A$ in $G$ with $\langle A\rangle$ not virtually
nilpotent has size at most $c^{-1/\delta}K^{2/\delta}$, thus our function
$f$ above will be polynomial in $K$. Moreover the application of Helfgott type
growth to infinite groups $G$ works well because we never have to worry
about the finite subset $A$ being ``most'' of $G$ and we do not have to stick
to the case when $A$ generates $G$. 

As for the degree of our polynomial function $f$, it is shown in
Proposition 4.1 using an elementary construction
that nearly all finitely generated
groups $G$ which are not virtually nilpotent
have finite subsets $A$ with $|A|\rightarrow\infty$ such that 
$\langle A\rangle$ is not virtually nilpotent but $|A^3|\le (2l+3)|A|^2$,
where $l$ is the number of generators for $G$. In
particular $f$ cannot be subquadratic here, as seen by taking values of
$K$ such that $|A^3|=K^2|A|$ whereupon $K^2\le (2l+3)f(K)$. 
Thus if $\delta=1$ holds in the above expression for some
group $G$ then it seems reasonable to say that $G$ has no decent sized
approximate groups, because any $K$-approximate group contained in $G$ is
either contained in a small (i.e. virtually nilpotent) subgroup of $G$ or
is only a constant times bigger than $K$-approximate groups formed using
basic constructions.

In this paper we show that a free product of two groups has no decent
sized approximate groups if the two factors do not. We also show that
limit groups have no decent sized approximate groups. More precisely, if
$G$ is a limit group and $A$ a finite subset then we show Helfgott type
growth
in that either $|A^3|\ge c|A|^{1+\delta}$ for $\delta=1$ (which is the
best possible value of $\delta$ here, see Proposition 4.1) or
$\langle A\rangle$ is free abelian. As for free products $\Gamma=G*H$
where $G$ and $H$ are arbitrary groups, we show that either 
$|A^3|\ge c|A|^{1+\delta}$ for $\delta=1$ (again best possible) or $A$ can
be conjugated into one of the factors of the free product, or 
$\langle A\rangle$ is infinite cyclic or infinite dihedral. We also make
the value of $c$ explicit and it is independent of the group, indeed we
take $c=1/7776$ throughout.

In \cite{saf} by Safin it was shown that in the free group $F_2$ of rank
2 (and hence in all free groups $F_k$ of rank $k$) we have $|A^3|\ge c|A|^2$
for all finite subsets $A$, unless $\langle A\rangle$ is infinite cyclic.
This was the first example of an infinite group $G$ with Helfgott type
growth where we can take $\delta=1$. Of course this will also be true
for subgroups of $G$ but a subgroup of a free group is free. In
\cite{saf} the constant $c$ was not given but is explicit when followed 
through the paper and that is what we do here. 
Our result mirrors this proof of Safin
closely by using the normal form for free products in place of words in
the standard generators for the free group $F_2$, but we have to work
harder at the beginning (in finding a suitable proportion of $A$ 
where growth should take place) and at the end (when identifying the
subgroup generated by those sets $A$ with small triple product, which
might not be cyclic).

The proof for free products is given in Section 2, then in Section 3 we
consider limit groups and other examples of infinite groups with Helfgott
type growth.
Finally in Section 4 we give examples of groups which are not virtually
nilpotent and which do not have Helfgott type growth for any positive
values of $c$ and $\delta$.
In particular there exist such examples which
are free products with amalgamation and also HNN extensions, formed 
using only the integers $\Z$. We end with a list of questions which is
designed to encourage the appearance of more groups without Helfgott type
growth, or at least without Helfgott type growth where $\delta=1$.

\section{Helfgott type growth in free products}

We begin with some standard facts on free products of groups, for which
see \cite{mks}, and periodic words following \cite{saf} (which in turn
followed \cite{raz}).

If $G$ and $H$ are any groups (implicitly assumed not to be the trivial
group $I=\{id\}$) then the {\bf free product} $G*H$ of the
{\bf factor groups} $G,H$ consists of $k$-tuples $(x_1,\ldots ,x_k)$ of
arbitrary length and where either the odd $x_i$s are taken from
$G-I$ and the even $x_i$s from $H-I$ or the other way around. This tuple
is referred to as the {\bf normal form} of an element $\gamma$
in $G*H$ but is really the definition of $\gamma$. If $k=1$ then we have
the element $(x)$ for $x\in G-I$ or $x\in H-I$ and we may omit the brackets.
The identity is represented by the empty tuple $\emptyset$ where $k=0$.

This value of $k$ is an important invariant of the element $\gamma\in G*H$
called the {\bf syllable length} $\sigma(\gamma)$. In this paper we
divide $(G*H)-I$ into four disjoint sets called {\bf types}. The idea is
that for any $\gamma\in (G*H)-I$ we need to take into account the parity of
$\sigma(\gamma)$ and whether the normal form of $\gamma$ starts with an
element in $G$ or an element in $H$ (whereupon we would know in which
factor the final entry lies). This gives rise to our four types of
{\bf $G$-even}, {\bf $H$-even}, {\bf $G$-odd}, {\bf $H$-odd} and the
notion of an {\bf odd} or {\bf even} element if we only require the
parity of $k$.  

Multiplication $xy$ of elements $x=(x_1,\ldots ,x_k)$ and
$y=(y_1,\ldots ,y_l)$ is given by first concatenating to get
$(x_1,\ldots ,x_k)(y_1,\ldots ,y_l)$. If $x_k$ and $y_1$ are in
different factors (for instance if $x$ is $G$-odd then this occurs
precisely when $y$ is $H$-even or $H$-odd) then we just remove the
interior brackets to obtain this $(k+l)$-tuple for $xy$. However if say
both $x$ and $y$ are $G$-odd then $x_ky_1$ is an element of $G$ and one
of two possibilities occurs. Either $x_ky_1\neq id$ in which case we say
that {\bf absorption} has occurred in forming the product $xy$ with
normal form
\[(x_1,\ldots ,x_{k-1},x_ky_1,y_2,\ldots ,y_l)\]
and syllable length $k+l-1$. Otherwise $x_ky_1=id$ and we say there is
{\bf cancellation} in $xy$. So
\[(x_1,\ldots ,x_{k-1}y_2,\ldots ,y_l)\] will be the normal form
of $xy$ if we now have absorption, that is if $x_{k-1}y_2\neq id$.
However we could have further cancellation and so the normal form is
found by continuing to cancel until we have absorption (or we run out
of letters in $x$ or $y$).

We now follow \cite{saf} in defining periodic words, but here adapted to
the free product case. A {\bf periodic} element $y$ of $G*H$ is one which
is of the form $y=\overline{y}^st$ where $s\ge 2$, $\overline{y}\in (G*H)-I$
and $t\in G*H$, no cancellation or absorption occurs anywhere in forming
this product, and the normal form of $t$ is a prefix of the normal form
for $\overline{y}$. In particular $\overline{y}$ must be an even element,
and if $\overline{y}$ is $G$-even say then $t$ is $G$-even or $G$-odd
(or empty). We assume throughout that $\overline{y}$ is not a proper power
of some other word, in which case we call $\overline{y}$ the
{\bf period} and $t$ the {\bf tail} of the periodic word $y$: these are
both uniquely defined. When $t$ is empty we say that $y$ is
{\bf totally periodic}.

We can also work from the right of the normal form, giving rise to the
definition of a {\bf right periodic} word (or totally right periodic if
$h$ is empty) $z=h\overline{z}^s$ where as before we have that
$\overline{z}$ is the (unique) {\bf right period} of $z$ and $h$ is a
(unique) suffix of $\overline{z}$ called the {\bf head} of $z$.

A word is periodic if and only if it is right periodic (explaining why we
can use the term periodic above rather than left periodic). Moreover on
taking two periodic words $y_1=\overline{y}^{s_1}t_1$ and
$y_2=\overline{y}^{s_2}t_2$ with the same period, suppose that they also have
the same right period. Then it follows that the tails are the same,
so $t_1=t_2$. (One can picture this by placing the entries of $\overline{y}$
clockwise in turn around a circle, starting from the top, so a word with
period $\overline{y}$ corresponds to a clockwise walk from the top that runs
at least twice around this circle. As $\overline{y}$ has no rotational
symmetry by minimality of the period, we see that if the right periods of
$y_1$ and $y_2$ coincide then the backwards walk of each word must have
started in the same place.)

Finally we need to adapt the idea of a periodic element in the case of
free products, in order to take account of odd elements. The problem is
that if $\overline{y}^st$ is a periodic word which is an odd element of
$G*H$ then absorption (or even cancellation) takes place when forming
its square which we need to control. Note that here the syllable length
must be at least 5 because $\overline{y}$ is even, $s\ge 2$ and
$t\neq\emptyset$, so the idea is that we do not worry about the first element
or last element of the normal form. We therefore need the notion of what we
call here an interior periodic element. This has exactly the same definition
as before for even elements. However we say that the $G$-odd
element $y$ is {\bf interior periodic} if there exist $g_1\in G$
and $g_2\in G-I$
such that $y=g_1xg_2$ where $x$ is $G$-even and periodic, so $x$ can
be written as $\overline{y}^st$ for $t$ a prefix of $\overline{y}$ and where
$\overline{y}$ and $t$ are both $G$-even elements (or $t$ is empty).
Thus $y=g_1\overline{y}^stg_2$
where absorption takes place between $g_1$ and $\overline{y}$
but no absorption or cancellation can occur otherwise in this
expression. We cannot have cancellation between $g_1$ and $\overline{y}$
because then $y$ would start with an element from $H$. Similarly
$g_2\neq id$ but we do allow $g_1$ to be the identity, thus interior
periodic does include the definition of periodic by taking $g_1=id$ and
$g_2$ to be the appropriate element of $G-I$. 

We have the corresponding definition of interior periodic words for
$H$-odd elements too.
We still call $\overline{y}$ the period and $t$ the tail of the interior
periodic element $y=g_1\overline{y}^stg_2$ as above (or tail $tg_2$ if
$tg_2$ is a prefix of $\overline{y}$)
and they are again uniquely defined. Similarly
we also have the concept of an {\bf interior right periodic} word, which
again has a well defined right period and head. 
As before, a word is interior periodic if and only if it is interior
right periodic and
if two interior periodic words have the same period and right period
then they have the same tail as well. These facts can be seen by removing 
the first and last letters from the normal form of $y$, 
resulting in a genuine periodic word unless $s=2$ and $t=\emptyset$,
but the arguments continue to hold in this case too.

We begin our proof by first adapting Lemma 1 of \cite{saf} to the free product 
case. Here we note the obvious but useful fact that 
if we are given a finite subset $A$ of any group $\Gamma$ then the
growth of $A$ is the same as the conjugate subset $\gamma A\gamma^{-1}$
for all $\gamma\in\Gamma$, as $|A^n|=|\gamma A^n\gamma^{-1}|=|(\gamma
A\gamma^{-1})^n|$.
\begin{lem}
Let $A$ be any subset of $\Gamma=G*H$. Then there exists a conjugate
subset $\gamma A\gamma^{-1}$ in $\Gamma$ of $A$ with the following
property: there is no element $x$ in $G-I$ or in $H-I$  
such that more than half of the elements in $\gamma A\gamma^{-1}$
have normal form $(x,\ldots ,x^{-1})$.
\end{lem}

\begin{proof}
Suppose there is such an $x=g$ in (without loss of generality) $G-I$,
so that more than half of the elements in $A$ are 
$G$-odd with normal form $(g,\ldots ,g^{-1})$. Then on forming the
conjugate subset $g^{-1}Ag$, these elements will have their
syllable length reduced by 2 (but see the note
below). Now the syllable length of any other element can go up by at
most 2 so the total syllable length of the elements in $A$ must decrease.
If the conclusion of the Lemma is still not satisfied then we can continue
the process of conjugating to reduce syllable length and it is clear that  
this process must eventually terminate.
\end{proof}
Note: There is one case where the conclusion of Lemma 2.1 fails. This is
if we have an element in normal form $(x,\ldots ,x^{-1})$ of syllable
length 1. Now this can only happen if the element itself is equal
to $x$ in $G-I$ (say) with $x$ of order 2, but on conjugating by $x^{-1}$
we have $x^{-1}xx$ is still of length 1. Thus it would be possible not to
decrease the total syllable length, but only if $|A|=2k+1$ for some $k$ and
where $k+1$ elements, including $x$, have normal form $(x,\ldots ,x)$
and the other $k$ elements are all $H$-odd. For instance consider the
infinite dihedral group $\langle x\rangle*\langle y\rangle$ with $x$ and
$y$ both of order 2, and
\[A=\{x,xyx,\ldots ,(xy)^kx,y,yxy,\ldots ,(yx)^{k-1}y\}.\]
However we only need the Lemma for Theorem 2.2 and in this case the 
conclusions are immediately satisfied by taking $X$ and $Y$ to be the
$G$-odd and the $H$-odd elements.

\begin{thm} For any finite set $A\subseteq G*H$ there exist (not
necessarily disjoint) subsets $X,Y$ of a conjugate in $G*H$ of $A$
such that $|X|,|Y|\ge (1/18)|A|$ and such that for any $x\in X$ and
$y\in Y$ there is no cancellation in $xy$ or in $yx$ (although there
may be absorption).
\end{thm}
\begin{proof} We assume by conjugating that $A$ satisfies the
conclusions of Lemma 2.1 and we first divide $A$ into the four
$G$-odd, $G$-even, $H$-odd and $H$-even subsets. 
If the identity is in $A$ then its type is not well defined, but we will
place it in either of the odd subsets because cancellation does not occur
when an element is multiplied by the identity. 

If either of
the two even subsets has size at least $|A|/18$ then we are happy
to take $X=Y$ to be that subset. We are also happy taking $X$ to be
the $G$-odd elements and $Y$ the $H$-odd elements if both have size
$|A|/18$. Without loss of generality this only leaves the case where
the set $G_{\mbox{odd}}$ of $G$-odd elements satisfies
$|G_{\mbox{odd}}|>(5/6)|A|$.
In this case we require a more delicate argument which involves
partitioning the elements in $G_{\mbox{odd}}$ into two subsets of
reasonable size in order to avoid cancellation, which will work
unless a large proportion of elements in $G_{\mbox{odd}}$ have
normal form $(g,\ldots ,g^{-1})$.

Consider the map $M$ from $G_{\mbox{odd}}$ to $G\times G$ given by
sending an element $(g,\ldots ,\gamma)$ in normal form to the 
ordered pair $(g,\gamma^{-1})$. This also makes sense for
elements in $G$, including the identity which maps to $(id,id)$.
Of course $G\times G$
could be infinite but the image $\mbox{Im}(M)$ is finite.
Let $E=\{g_1,\ldots ,g_n\}$ be the set of elements of $G$ that
appear either as a first or second entry of $\mbox{Im}(G)$. 
We have that $n\ge 2$ as not all
elements in $G_{\mbox{odd}}$ are of the form $(g,\ldots ,g^{-1})$ by
Lemma 2.1. An obvious way to form two subsets of $G_{\mbox{odd}}$
without cancellation between them is to divide $E$ into two disjoint
subsets $E_1\cup E_2$ of roughly equal size, say
$E_1=\{g_1,\ldots ,g_{\lceil n/2\rceil}\}$ and $E_2=
\{g_{\lceil n/2\rceil +1}, \ldots ,g_n\}$. This gives a few possibilities
for $X$ and $Y$; namely we could set $X=Y$ to be the inverse image
under $M$ of $E_1\times E_2$, or of $E_2\times E_1$, or we could
set $X$ to be the inverse image of $E_1\times E_1$ and $Y$ that of
$E_2\times E_2$. We are done if in any of these three cases we have
that both $|X|$ and $|Y|$ are at least $|G_{\mbox{odd}}|/15>|A|/18$.
Thus we could only fail here if the inverse images of $E_1\times E_2$,
of $E_2\times E_1$ and (without loss of generality) $E_2\times E_2$
is less than $|G_{\mbox{odd}}|/15$. This forces 
$M^{-1}(E_1\times E_1)>(1-3/15)|G_{\mbox{odd}}|$.

In this case we transfer elements one by one from $E_1$ to $E_2$,
thus decreasing $M^{-1}(E_1\times E_1)$ and increasing
$M^{-1}(E_2\times E_2)$, although $M^{-1}(E_1\times E_2)$ and
$M^{-1}(E_2\times E_1)$ can go either way. We stop as soon as
one of these last three quantities is at least $|G_{\mbox{odd}}|/15$
and we are done in the final two cases. We are also done when
$M^{-1}(E_2\times E_2)\ge |G_{\mbox{odd}}|/15$, which will occur at
some point, unless $|M^{-1}(E_1\times E_1)|$ is now less than
$|G_{\mbox{odd}}|/15$. If so, suppose we have just moved $g\in E$, where
our partition previously was $E=F_1\cup F_2$ and now is $E=E_1\cup E_2$
so that $F_1=E_1\cup\{g\}$ and $E_2=F_2\cup\{g\}$. On doing this we
have $|M^{-1}(F_1\times F_1)|>(4/5)|G_{\mbox{odd}}|$ but
$|M^{-1}(E_1\times E_1)|$ has dropped down to below 
$|G_{\mbox{odd}}|/15$, which implies
$M^{-1}(F_1\times F_1-E_1\times E_1)>(11/15)|G_{\mbox{odd}}|$. Now
$F_1\times F_1-E_1\times E_1=(\{g\}\times E_1)\cup
(E_1\times\{g\})\cup\{(g,g)\}$ and $M^{-1}(\{g\}\times E_1)$ is a subset
of $M^{-1}(E_2\times E_1)$. Thus if the former of these two sets
has size at least $(1/15)|G_{\mbox{odd}}|$ then we are done on taking
$X=Y=M^{-1}(E_2\times E_1)$. Similarly for $E_1\times\{g\}$ so we are
fine unless $|M^{-1}(\{(g,g)\})|>(9/15)|G_{\mbox{odd}}|>|A|/2$. 
But this implies
that more than half the elements of $A$ have normal form
$(g,\ldots ,g^{-1})$, which was eliminated at the start of the proof
by taking a suitable conjugate.
\end{proof}

\begin{lem}
For any finite set $A\subseteq\Gamma=G*H$, there exist subsets $X$ and
$Y$ of a conjugate of $A$ such that $|X|,|Y|\ge (1/36)|A|$, with no
cancellation in $xy$ or $yx$ for all $x\in X,\, y\in Y$ and such that
the syllable length $\sigma(x)\le \sigma(y)$ for all $x\in X$ and $y\in Y$.
\end{lem}
\begin{proof}
Take $X$ and $Y$ as in Theorem 2.2 but notice they can be swapped with
the same conclusion holding. Let the elements $x_i$ and $y_i$ of each of
these sets be placed in ascending order of syllable length and set
$\sigma_X$ to be the syllable length of the median element
$x_{\lfloor(|X|+1)/2\rfloor}$, and the same for $\sigma_Y$. If 
$\sigma_X\le\sigma_Y$ then take the first $\lfloor\frac{|X|+1}{2}\rfloor$
elements for the new set $X$ and the last $\lfloor\frac{|Y|+1}{2}\rfloor$
 for the new set $Y$. If $\sigma_X>\sigma_Y$ then swap $X$ and $Y$ before
doing the same.
\end{proof}

The idea now, as based on Lemma 3 in \cite{saf}, 
is to take an element $a\in A$ (usually $a$ will be in $Y$)
and consider the map $F_a$ from
$X\times X\rightarrow A^3$ given by $(x_1,x_2)\mapsto x_1ax_2$. If this
map is injective or even at most 2--1 then we have $|A^3|\ge (1/2)|X|^2$,
so we examine when this fails with the aim of finding severe restrictions
on $X$ and $Y$. The fact that cancellation is not taking place between
elements of $X$ and elements of $Y$ means that periodicity plays a big part
in such restrictions, but we first need to eliminate very short words. 

\begin{lem} If $X$ and $Y$ are in Lemma 2.3 then either $|A^3|\ge (1/4)|X|^2$
or we can assume that all elements in $Y$ have syllable length at least 4.
\end{lem}
\begin{proof} Recall that however $X$ and $Y$ were obtained
in Theorem 2.2 and Lemma 2.3, all elements in $X\cup Y$ have the same
parity, so we first deal with odd elements.
Suppose there is an element in $Y$ with syllable length
1. Then $\sigma(x)=1$ (or 0)
for all $x\in X$; that is (without loss of generality)
$X\subseteq G$. If there is any $a\in A$ which is not in $G$ then the
map $F_a$ is injective on $X\times X$, giving $|A^3|\ge |X|^2$.
This is because even though there
might be absorption, we can read off $(x_1,x_2)$ from the front and back
of $x_1ax_2$. If however $a\in G$ then we have absorption both before and
after $a$, but now we are in the case where (after an initial
conjugation) $A\subseteq G$.

Now suppose we have $y\in Y$ with $\sigma(y)=3$. Then $\sigma(x)=$(0 or)
1, or 3
for all $x\in X$ and so one of these will occur for at least $|X|/2$
elements. Again the map $F_y$ will be injective on this subset of $X$
because we may have absorption but do not have cancellation and we know
the syllable length of the elements, so can recreate a pair of elements
from their image. This gives $|A^3|\ge (|X|/2)^2$.

The argument is the same when all elements are even by replacing syllable
lengths 1 and 3 with 2 and 4 respectively. Here we do not have to worry
about $A$ being conjugate into a factor.   
\end{proof}

As we now assume that all elements in $Y$ have syllable length at least
4, we can start to look for periodic words.

\begin{lem}
Let $X$ and $Y$ be as in Lemma 2.3 and such that 
all elements of $Y$ have syllable length at least 4. Take any $y\in Y$ 
and consider the map
\[F_y:X\times X\rightarrow AyA\mbox{ given by }F_y(u_1,u_2)=u_1yu_2.\]
If a point in $AyA$ has at least
three preimages then $y$ is interior periodic.  
Moreover when $X\cup Y$ does not consist solely of odd elements of the
same type, we have that $y$ is periodic with period $\overline{y}$
and one of the first components
in these three preimages has normal form ending with $\overline{y}$.

In the case when all elements in $X\cup Y$ are odd elements of the
same type, say $G$-odd without loss of generality, then 
there is $s\ge 2$ and $g\in G$, $\gamma\in G-I$ such that 
$y=g\overline{y}^st\gamma$ 
for period $\overline{y}$ (a $G$-even element)
and $t$ a prefix of $\overline{y}$ which is either empty or also 
$G$-even. We also have that one of the first components in these three
preimages ends with either $\overline{y}g^{-1}$ 
or $\omega g\overline{y}g^{-1}$,
where $\omega$ is the final entry in one of the other two first components.
\end{lem}
\begin{proof}
Suppose $u_1yu_2=v_1yv_2=w_1yw_2$. We will assume that all elements
are $G$-odd with absorption but not cancellation taking place.
This is because otherwise the proof is considerably easier in that we
ignore all reference to absorption, making it close to the original proof
of Lemma 3 in \cite{saf} for free groups.
 
We set
\[u_1=(a_1,b_1,\ldots ,b_{j-1},a_j),\,
v_1=(c_1,d_1,\ldots ,d_{k-1},c_k),\,
w_1=(\gamma_1,\delta_1,\ldots ,\delta_{l-1},\gamma_l)\]
and
\[y=(g_1,h_1,\ldots ,h_{N-1},g_N)\mbox{ for } a_i,c_i,\gamma_i,g_i\in G
\mbox{ and } b_i,d_i,\delta_i,h_i\in H.\]
Let us first assume that $j=k$, so that
$u_1$ and $v_1$ have the same syllable length $2k-1$. 
As $N>1$ (because of the condition on syllable lengths of $Y$)
and there is only absorption at each end, we can equate
the two normal forms for $u_1yu_2=v_1yv_2$ to obtain
$a_i=c_i$ and $b_i=d_i$. Thus $u_1=u_2$ and $v_1=v_2$ anyway.

The same also holds if any two of $j,k,l$ are equal so we now
suppose that $2j-1=\sigma(u_1)<2k-1=\sigma(v_1)<2l-1=\sigma(w_1)
\le 2N-1=\sigma(y)$.
On comparing respective entries for the element $u_1yu_2=v_1yv_2$,
we have $c_kg_1=g_{k-j+1}$ and then the letters start to reoccur, so that
$h_1=h_{k-j+1},h_2=h_{k-j+2},\ldots$ which makes the sequence $h_i$
repeat after the first $k-j$ elements. We also have the same property
for the sequence $g_i$ except that $c_kg_1=g_{k-j+1}$ and $g_{N-k+j}
=g_N\epsilon$, where $\epsilon\in G$ is the first element in the normal form
of $u_2$. This means that the sequence
\[x=(c_kg_1,h_1,g_2,\ldots ,h_{N-1},g_N\epsilon)\]
of length $2N-1$ repeats after the first $2(k-j)$ elements. The same is
true on replacing $j$ with $k$, $k$ with $l$, $c_k$ with $\gamma_l$, and
$\epsilon$ with the first element $\alpha$
in the normal form for $v_2$.

Now if $2(k-j)$ and $2(l-k)$ are both at least $N$ then $2(l-j)\ge 2N$,
but we know $1\le j,l\le N$ which is a contradiction. Thus suppose that
$2(k-j)<N$. Then $x$ is periodic so we can write $x=\overline{y}^s\tau$ for
$s\ge 2$ and with no cancellation, where $\tau$ is a prefix
of $\overline{y}$. We also have that $\sigma(\overline{y})$ divides
$2(k-j)$. Now $\overline{y}$ is a $G$-even element
and $\tau$ is $G$-odd. 
Consequently we can say that $y=c_k^{-1}x\epsilon^{-1}$,
where we have absorption, and we take $g=c_k^{-1}$ in the statement of the
Lemma. The same holds if
$2(l-k)<N$ except that now $y=\gamma_l^{-1}x\alpha^{-1}$ and 
$g=\gamma_l^{-1}$, thus either
way we conclude that the element $y$ is interior periodic. 

For the final part, first suppose that 
$2(l-k)\le 2(k-j)$. We then compare the entries of 
$w_1yw_2$ and $v_1yv_2$ from the $(2k-1)$th
place to the $(2l-1)$th place. This tells us that the element $w_1$
is such that
\[(\ldots,\gamma_k,\delta_k,\ldots ,\gamma_{l-1},\delta_{l-1},\gamma_l)
=(\ldots , c_kg_1,h_1,\ldots ,g_{l-k},h_{l-k},g_{l-k+1}g_1^{-1}).\]
As $\gamma_l=g^{-1}$, we have that the normal form 
of $w_1$ ends in $\overline{y}g^{-1}$ when $2(l-k)$ is a 
proper multiple of the period of $x$. However if the period is exactly
$2(l-k)$ then $\overline{y}$ starts with the element $\gamma_lg_1$, so
$\gamma_k$ is equal to $c_kg(\gamma_lg_1)$. 

Otherwise we have $2(k-j)\le 2(l-k)$and we replace $w_1yw_2$ and $v_1yv_2$
in the above argument with $v_1yv_2$ and $u_1yu_2$ respectively.
\end{proof}

Our aim now is to show that we have Helfgott type growth of the triple product
$A^3$, unless all elements in $Y$ are (interior) 
periodic with the same period.
\begin{co}
Let $X$ and $Y$ be as in Lemma 2.5. If there exists $y\in Y$ which is
not interior periodic, or $y$ is interior periodic and equal to 
$\overline{y}^st$ (for $X\cup Y$ not all odd elements of the same type) 
or $g\overline{y}^st\gamma$ (when $X\cup Y$ is $G$-odd) and 
there is no element in $X$ which ends
$\overline{y}$ (in the first case) or $\eta\overline{y}g^{-1}$ for
some $\eta\in G$ (in the second case)
then $|A^3|\ge (1/2592)|A|^2$.
\end{co}
\begin{proof}
On taking such a $y$, we have by Lemma 2.5 that the map $F_y$ is at most
2 to 1, so $|A^3|\ge (1/2)|X|^2$ and $|X|\ge (1/36)|A|$.
\end{proof}
\begin{prop}
Let $X$ and $Y$ be as in Lemma 2.5. Then either the triple product set
$XYX$ has size at least $(1/7776)|A|^2$ or all elements of $Y$ are interior
periodic with the same period.
\end{prop}
\begin{proof}
Again we assume that all elements are $G$-odd, with the other case being
covered by replacing interior periodic with periodic. We suppose that
we have two elements $y_1$ and $y_2$ in $Y$ 
such that $y_1=g_1\overline{y_1}^{s_1}
t_1\gamma_1$ and $y_2=g_2\overline{y_2}^{s_2}t_2\gamma_2$ for $s_1,s_2\ge 2$.

If there is a
subset $B_1$ of $X$ containing no elements ending in $g_1^{-1}$
and such that $|B_1|\ge 1/2|X|$ then we can apply Lemma 2.5 to the map
$F_{y_1}$ but with domain restricted to $B_1\times X$, whereupon we
conclude that $|A^3|\ge (1/2)|B_1||X|\ge (1/4)|X|^2\ge (1/4)(|A|/36)^2$.
The same statement holds for the equivalent subset $B_2$
so we are left with over half of the elements in $X$ ending in
$g_1^{-1}$, and over half ending in $g_2^{-1}$. 
Consequently there is something in the intersection
which means that in fact $g_1=g_2=g$ say. We now 
assume that the periods $\overline{y_1}$ and $\overline{y_2}$
are distinct and define similar sets
$S_1$ and $S_2$ where
\[S_1=\{x\in X: x\mbox{ does not end }\eta\overline{y_1}g^{-1}
\mbox{ for some }\eta\in G\}\]
and the same for $S_2$. Now in a similar fashion to before, we have that
if $|S_1|\ge (1/3)|X|$ then we can again apply Lemma 2.5 to the map
$F_{y_1}$ but this time with domain $S_1\times X$, thus
$|A^3|\ge (1/2)|S_1||X|\ge (1/6)|X|^2\ge (1/6)(|A|/36)^2$.
The equivalent statement holds for $S_2$
so we are left with the complement $S_1'$ of $S_1$ in $X$ having
$|S_1'|>(2/3)|X|$ and where every element of $|S_1'|$
ends $\eta\overline{y_1}g^{-1}$ for some $\eta\in G$. Now it appears
that $\eta$ varies over $S_1'$ but another appeal to Lemma 2.5 using
$F_{y_1}$ with domain $S_1'\times X$ tells us that either
$|A^3|\ge (1/2)|S_1'||X|\ge (1/3)|X|^2$ or something in $|S_1'|$ ends in
either $\overline{y_1}g^{-1}$ or $\omega g\overline{y_1}g^{-1}$, where
$\omega$ is the last entry in some other element in $S_1'$. But this
is always $g^{-1}$ thus everything in $S_1'$ does in fact end
$\overline{y_1}g^{-1}$.

The same argument tells us that the complement $S_2'$ of $S_2$ consists
solely of elements ending $\overline{y_2}g^{-1}$, and both $S_1'$ and
$S_2'$ consist of over two thirds of $X$.
Thus the intersection is not empty, but now this
means that one of these sets will be contained in the other.
Thus we have a set $S$ consisting of
over two thirds of $X$ where every element in $S$ ends with both these 
expressions. Consequently we may as well conjugate by $g$ so that
every element in $S$ now ends with both $\overline{y_1}$ and $\overline{y_2}$,
and with the new $Y$ containing $\overline{y_1}^{s_1}t_1\gamma_1$ and
$\overline{y_2}^{s_2}t_2\gamma_2$, where we have changed $\gamma_i$ by
postmultiplying with $g$.

We have now reached a position where we can essentially follow \cite{saf}
Lemmas 2, 5 and 6 for the remainder of this proof. As $\overline{y_1}\neq
\overline{y_2}$, we cannot have $\sigma(\overline{y_1})=
\sigma(\overline{y_2})$ so say $\sigma(\overline{y_1})<
\sigma(\overline{y_2})$ without loss of generality. Let $T$ be the
subset of (this new)
$S$ consisting of words in $S$ which end $\overline{y_2}^2$.
However the elements in $T$ also end in $\overline{y_1}$. Take some
$x\in T$ and let $n\ge 1$ be such that $x$ ends in $\overline{y_1}^n$
but not in $\overline{y_1}^{n+1}$. Now take some other $x' \in T$. 
If $\sigma(\overline{y_1}^n)<\sigma(\overline{y_2}^2)$ then $x'$ also
ends in $\overline{y_1}^n$. However we cannot have 
$\sigma(\overline{y_1}^n)\ge\sigma(\overline{y_2}^2)$ because then
$\overline{y_2}^2$ also ends in $\overline{y_1}^2$ at least, and this
element is right periodic with two different right
periods $\overline{y_1}$ and $\overline{y_2}$. However we can now swap
the roles of $x$ and $x'$ to conclude that all elements of $T$ end
in $\overline{y_1}^n$ and none end $\overline{y_1}^{n+1}$.

We use Lemma 2.5 again (but now in the case where absorption does not
occur) with $Y$ replaced by the equal sized set
$\overline{y_1}^nY$ and we take the map $F_{v_1}$ where
$v_1=\overline{y_1}^ny_1$, with
domain $T\overline{y_1}^{-n}\times X$. As nothing in the first factor 
ends in $\overline{y_1}$, we have that 
$|(T\overline{y_1}^{-n})v_1X|\ge (1/2)|T||X|$. However
we now do the same with $Y$ replaced by $\overline{y_2}Y$ and where the 
map is now $F_{v_2}$ for $v_2=\overline{y_2}y_2$ with
domain $T'\overline{y_2}^{-1}\times X$, where $T'$ is 
the complement of $T$ in $S$.
Lemma 2.5 also tells us that $|T'\overline{y_2}^{-1}v_2X|\ge (1/2)|T'||X|$
because no word in $T'\overline{y_2}^{-1}$ ends in $\overline{y_2}$
and $v_2$ has period $\overline{y_2}$.
Hence $|Ty_1X|\ge (1/2)|T||X|$ and $|T'y_2X|\ge (1/2)|T'||X|$. Now we must 
have either $T$ or $T'$ having size at least $|S|/2$, with both subsets
contained in $X$ and $y_1,y_2\in Y$ thus $|XYX|\ge
(1/4)|S||X|\ge (1/6)|X|^2\ge (1/7776)|A|^2$.
\end{proof}

\begin{co}
Let $X$ and $Y$ be as in Lemma 2.5. Then either the triple product set
$XYX$ has size at least $(1/7776)|A|^2$ or every $y$
in $Y$ is interior periodic of the form $\overline{y}^st$ (when $X\cup Y$
is not all odd elements of the same type) or $g\overline{y}^st\gamma$ 
(when $X\cup Y$ is $G$-odd),
where $g,\overline{y},t,\gamma$ are all independent of $y$. 
\end{co}
\begin{proof} As usual we only treat the second case.
By Proposition 2.7 we have that if the inequality fails then all elements
of $Y$ are of the form $g\overline{y}^st\gamma$ for fixed 
$\overline{y}$ and the proof gives that $g$ is fixed too,
although $s\ge 2,t$ and $\gamma$ can vary across $Y$.
But we can now run the whole argument from Lemma 2.5 in the opposite
direction, meaning that we examine the end of elements in $Y$ and the
beginning of elements coming from the second factor of $X$ in the
domain of $F_y$. This will tell us that either the appropriate inequality
is satisfied or all elements in $Y$ are interior right
periodic, so that they are of the form $g_0r\overline{z}^s\gamma_0$, where
now $\gamma_0$ and the right period $\overline{z}$ 
are independent of $y$, although $s\ge 2$, the suffix $r$ of $\overline{z}$ 
and $g_0$ can vary with $y$. Putting these facts
together gives us our conclusion. 
\end{proof}

We now have enough information to turn interior periodic elements into
elements that are genuinely periodic.
\begin{lem}
Let $X$ and $Y$ be as in Corollary 2.8, with all elements of $X\cup Y$ being
$G$-odd. Then either the triple product set
$XYX$ has size at least $(1/7776)|A|^2$, or $|Y^3|\ge (1/1296)|A|^2$, 
or we can conjugate $A$ so that
all elements $y$ of $Y$ are periodic of the form $\overline{y}^st$,
for period $\overline{y}$ and tail $t$ independent of $y$. 
\end{lem}
\begin{proof}
By applying Corollary 2.8 we can assume that all elements $y_i$ in $Y$
are equal to $g\overline{y}^{s_i}t\gamma$. 
Let $\overline{y}=(g_1,h_1,\ldots ,g_N,h_N)$
and $t=(g_1,h_1,\ldots ,g_k,h_k)$ for $0\le k\le N-1$.
We now consider $|Y^2|$ directly and we have
$|A^3|\ge |Y^3|\ge|Y^2|$. 
On taking any $y_i,y_j\in Y$ and forming $y_iy_j$ 
we look for the $\gamma gg_1$ term in the product. 
This spoils the periodicity, allowing
us to recover $s_i,s_j$ then $y_i$ and $y_j$ which implies that
$|Y^2|\ge |Y|^2$, unless one of two cases occurs.
The first is that $\gamma gg_1=g_{k+1}$ so that the periodicity is not
broken at that place. But it will be in the very next entry, unless
$h_1=h_{k+1}$, and then we require $g_2=h_{k+2}$ and so on. On repeating
this argument through the whole period, we end up with $g_1=g_{k+1}$ as
well thus implying that $\gamma g$ is the identity. Now $\gamma$ and
$g$ are constant so all
elements in $Y$ are of the form $g\overline{y}^stg^{-1}$ with only $s$
varying. 

The other case is if there were cancellation completely, so that
$\gamma gg_1=id$. But then the previous term is now $h_kh_1$, again
spoiling the periodicity which should give the term $h_k$ unless this
cancels too. Again we repeat this argument, requiring cancellation at
every stage until we have gone backwards through the whole period. Then
the cancellation required at this point tells us that $g_{k+1}g_1=id$.
But putting this together with $\gamma gg_1=id$ means that for 
$y=g\overline{y}^st\gamma$ we have $g^{-1}yg=\overline{y}^s\tau$ where
$\tau=(g_1,h_1,\ldots ,g_k,h_k,\gamma g)$. Now $\gamma g=g_{k+1}$ so
$\tau$ is a prefix of $\overline{y}$.
\end{proof}

We are now able to characterise those subsets with a small triple product,
although this is harder than in the free group case because we must deal
with infinite dihedral groups. We now start to consider the subset
$YAY$ of $A^3$ and we will first assume that the elements of $Y$ have 
empty tail.

\begin{thm} Suppose that $Y$ is a subset of $\Gamma=G*H$ 
such that all elements $y_i$ of $Y$ are
totally periodic with the same period,
so of the form $\overline{y}^{s_i}$ for $s_i\ge 2$.
Then for any $a\in\Gamma$ either $|YaY|=|Y|^2$ or 
$a$ and $Y$ together generate an
infinite cyclic or infinite dihedral subgroup of $\Gamma$.
\end{thm}
\begin{proof}
Let us assume all elements of $Y$ are $G$-even and we will set
$\overline{y}=(g_1,h_1,\ldots ,g_N,h_N)$.
If $a$ is $G$-even then there is no cancellation
or absorption in forming $y_iay_j$, so we have that $|YaY|\ge |Y|^2$ by
recognising the periodicity, unless $a$ is itself periodic with
period $\overline{y}$ (or $a$ is either $\overline{y}$ or $\emptyset$).
But we can assume $a$ is not in the cyclic subgroup
generated by $\overline{y}$, or else $\langle a,Y\rangle$ is an infinite
cyclic subgroup.

Now suppose that $a$ is $G$-odd. We will adopt a similar argument to that
in the last lemma, where we look for periodicity to establish injectivity
of $(y_i,y_j)\mapsto y_iay_j$. We set
$a=(a_1,b_1,\ldots ,b_{M-1},a_M)$, so
the $a_i$ are in $G-I$ and the $b_i$ in $H-I$. We begin by supposing that
$a_1\neq g_1$ so that the periodicity is broken in the very first place 
of $a$. Then on forming $\overline{y}^{s_i}a\overline{y}^{s_j}$ for various 
integers
$s_i$ and $s_j$, we see that we can again recover $s_i$ and $s_j$ by looking
for $a_1$ unless either
of two possibilities occur. The first is that there is cancellation on the
right of $a$ when calculating 
$\overline{y}^{s_i}a\overline{y}^{s_j}$ such that all
entries of $a$ are removed except $a_1$, whereupon absorption
takes place with the relevant entry of $\overline{y}^{s_j}$, 
changing $a_1$ into
$g_1$ and with the uncancelled part of $\overline{y}^{s_j}$ being
$(h_1,g_2,\ldots ,h_N)$ followed by $\overline{y}$ to some smaller power.
But this would require that $a_1$ times the first entry of $\overline{y}$
is $g_1$, implying that $a_1$ is the identity.

The other possibility is that $a_1$ cancels completely with the appropriate
entry $g_k$ of $\overline{y}^{s_j}$. In fact
this can occur but now the start of
$\overline{y}^{s_i}a\overline{y}^{s_j}$ reads 
$\overline{y}^{s_i-1}(g_1,h_1,\ldots ,
g_N, h_Nh_k)$ for some $k$ between 1 and $N$. 
This allows us again to detect periodicity unless there is
further cancellation between $h_N$ and $h_k$. But this argument can be
repeated, thus we require further cancellation between $g_N$ and $g_{k+1}$,
$h_{N-1}$ and $h_{k+1}$ (where all subscripts are taken modulo $N$),
and so on until we have run through a period
whereupon $g_1$ and $g_{N+k}=g_k$ cancel too. But this means that $a_1$ is
equal to $g_1$ after all.

We can now argue that if $M>1$ then $b_1=h_1$ as well, by
considering $y_iay_j=g_1\overline{z}^{s_i}(b_2,\ldots ,b_{M-1})a_Mg_1
\overline{z}^{s_j}g_1^{-1}$, where $\overline{z}=(h_1,g_2,\ldots ,h_N,g_1)$.
The $b_2$ entry destroys periodicity unless it disappears or is changed, so
some cancellation is needed in the above expression for $y_iay_j$.
But this can only occur initially between $a_M$ and $g_1$ 
(apart from at the very end which merely takes away the final $g_1$), 
and on removing $a_Mg_1$ we are back in the
above case on swapping $G$ and $H$.

Consequently we can continue this argument to conclude that
if $a$ is a $G$-odd element of $A$ such that   
$|YaY|<|Y|^2$ then $a=\overline{y}^kt$ for some
$k\ge 0$ and $t$ equal to a prefix $(g_1,h_1,\ldots, h_{n-1},g_n)$
of $\overline{y}$ for some $1\le n\le N$. 
We may as well assume here that $a=t$ 
because removing powers of $\overline{y}$ from the front
of $a$ will not change the group $\langle\overline{y},a\rangle$.

Next we move on to when $a$ is $H$-even and we will show that $|YaY|<|Y|^2$
can only occur when $a$ is a negative power of $\overline{y}$.
If $a$ begins with $h_N^{-1}$ then we can conjugate to get 
$h_N\overline{y}^{s_i}h_N^{-1}=(h_N,g_1,\ldots ,h_{N-1},g_N)^{s_i}$ for the
elements of $Y$ and replace $a$ with $h_Nah_N^{-1}$, so now the elements
in $Y$ are all $H$-even and $a$ is $G$-even thus we are back in the same
position on swapping $G$ and $H$. Of course there could
be further cancellation but at some point this process must stop if
$a$ is not a power of $\overline{y}$. Now that there is no cancellation
in forming $\overline{y}a$ for the new $\overline{y}$ and $a$, we can
follow the same argument as for $G$-odd elements above where we look to see
whether the first element of $a$ spoils the periodicity. This time we  
find that either the first entry of $a$ cancels with $h_N$ in the
absorption case or is the identity in the cancellation case. But
both of these are contradictions.

The final type to consider is when $a$ is $H$-odd, but this is the same
argument as the $G$-odd case, only with the words running the other way
when $G$ and $H$ are swapped. 
Consequently we conclude that we can have $|YaY|<|Y|^2$ here but only when
$a=r\overline{y}^k$ for some
$k\ge 0$ and $r$ equal to a suffix $(h_n,g_{n+1},\ldots,g_N,h_N)$
of $\overline{y}$ for some $1\le n\le N$. However in this case we have
that $\langle \overline{y},a\rangle=\langle \overline{y},t\rangle$ for
$t=(g_1,h_1,\ldots, h_{n-1},g_n)$

Thus if $a$ is such that $|YaY|<|Y|^2$ but
$a\notin\langle\overline{y}\rangle$ then we have seen that we can
take $a=(g_1,h_1,\ldots, h_{n-1},g_n)$ in all cases.
But from before we must also have cancellation
between $\overline{y}^{s_i}a$ and $\overline{y}^{s_j}$ of at least a complete
period. This gives us that
\[(h_n,g_{n+1},\ldots ,g_N,h_N,g_1,\ldots, h_{n-1},g_n)\mbox{ and }
(g_1,h_1,\ldots ,g_N,h_N)\]
cancel completely, so are inverses (where all subscripts are modulo $N$).
Let us set $(x_1,x_2,\ldots ,x_{2N-1},x_{2N})=(g_1,h_1,\ldots ,g_N,h_N)$
so that we have
\[(x_{2n},x_{2n+1},\ldots ,x_{2N-1},x_{2N},x_1,\ldots, x_{2n-2},x_{2n-1})
=(x_{2N}^{-1},x_{2N-1}^{-1},\ldots ,x_2^{-1},x_1^{-1})\]
This implies that $x_n$ and $x_{N+n}$ (where we now work modulo $2N$)
are self inverse elements so both are of order 2, and we 
have $x_i=x_{2n-i}^{-1}$ otherwise. 
Now we are interested in the group generated by $t$ and
$t^{-1}\overline{y}$ and these two elements are
$(x_1,x_2,\ldots ,x_{2n-1})$ and $(x_{2n},x_{2n+1},\ldots ,x_{2N})$. Under
our equations, we see that the first element is a conjugate of $x_n$ and
the second of $x_{N+n}$. Thus we have an infinite group generated by two
elements of order 2, which must therefore be the infinite dihedral group.
\end{proof}

We are not quite done for the case when the elements of $Y$ are all
totally periodic, because although we have obtained
our conclusion for the group generated by $Y$ and any $a\in A$, we
do not yet know what happens when we throw in all elements of $A$. 

\begin{prop} Suppose that $Y$ is a subset of $\Gamma=G*H$ 
such that all elements of $Y$ are
totally periodic with the same period.
Then for $S$ any finite subset of $\Gamma=G*H$, we have that
either there exists $a\in S$ with $|YaY|=|Y|^2$ or 
$\langle S,Y\rangle$ is an
infinite cyclic or infinite dihedral group.
\end{prop}
\begin{proof}
We are done by the proof of
Theorem 2.10 apart from one case. This is when $Y$ is
contained inside the cyclic group $C=\langle\overline{y}\rangle$ for
the even element $\overline{y}=(x_1,x_2,\ldots ,x_{2N})$ and there is
always $n$ between 1 and $N$ such that
each $a$ in $S-C$ is of the form
$(x_1,x_2,\ldots ,x_{2n-1})$ without loss of generality. Now $n$ could well
vary with $a$, but whenever this occurs we have $x_i$ and $x_{2n-i}$ are
inverse pairs for all $i$ between 1 and $2N$. If we think of a regular
$2N$-gon $P$
having vertices labelled $x_1,x_2,\ldots ,x_{2N}$ in order then any element
$a$ defines an axis through the opposite vertices $x_n$ and $x_{N+n}$ of
$P$, and pairs of elements $x_i,x_{2n-i}$ which are swapped by
reflection in this axis are inverse to each other.

Now let us take all such axes obtained from the various $a$ and consider
the dihedral group $D$ thus generated. The reflections in the first two axes,
through (say)
the vertices $x_{k},x_{N+k}$ and $x_{k+r},x_{N+k+r}$ where $r$ divides $N$,
generate the
whole of $D$. In particular consider the elements 
$\alpha=(x_1,x_2\ldots ,x_{2k-1})$ and 
$\beta=(x_1,x_2,\ldots ,x_{2k+2r-1})$. By the dihedral symmetry
these are conjugates of $x_k$ and $x_{k+r}$ respectively and so are
both of order two. Now any $a$ of the form $(x_1,\ldots ,x_{2n-1})$ must be
represented by an axis in $D$ and so $n=k+mr$ for some integer $m$ depending
on $n$. But as $x_i$ remains the same element under adding multiples of
$2r$ to the suffix, because this corresponds to a rotational symmetry
which is the product of two reflections and so the inverse has been taken
twice, we have that all the elements ${\bf x_m}=
(x_1,\ldots ,x_{2k+2mr-1})$ represented by axes
are given
by $\alpha,\beta,\beta\alpha\beta,\beta\alpha\beta\alpha\beta,\ldots$ for
$m=0,1,2,3,\ldots$. Now $(\beta\alpha)^{N/r}=\overline{y}$ so that
all of $S$ and $Y$ is contained inside the infinite dihedral group
$\langle\alpha,\beta\rangle$. Moreover
any subgroup of the infinite dihedral group is itself infinite dihedral
or infinite cyclic.
\end{proof}

We can now finish our main result.   
\begin{co}
Given any finite set $A$ of the free product $\Gamma=G*H$, we have that
either $A$ is conjugate into one of the factors, or 
$|A^3|\ge (1/7776)|A|^2$, or $\langle A\rangle$ is infinite cyclic
or infinite dihedral.
\end{co}
\begin{proof}
We follow through the results of this section, first applying Theorem 2.2
and Lemma 2.3 to obtain subsets $X,Y$ of (a conjugate of) $A$. 
Then Lemma 2.4
tells us that either $|A^3|\ge (1/5184)|A|^2$, or $A$
is conjugate into one of the factors, or we 
can apply Lemma 2.5, Corollary 2.6, Proposition 2.7, Corollary 2.8
and Lemma 2.9 to $X$ and $Y$. 
We conclude that either $|A^3|\ge (1/7776)|A|^2$ or we can 
further conjugate $A$ so that all elements of $Y$ are 
periodic with the same period $\overline{y}$ and tail $t$.
If $t$ is empty then Proposition 2.11 gives us that either
$|A^3|\ge (1/1296)|A|^2$ or $\langle A\rangle$ is infinite cyclic
or infinite dihedral.

For the case where $Y=\{\overline{y}^{s_i}t\}$ with tails, 
let us set $E=\{\overline{y}^{s_i}\}$ without the tail. Thus
$Y=Et$ and $|E|=|Y|$. First note that the elements in
$Y^2$ are of the form $\overline{y}^{s_i}t\overline{y}^{s_j}t$, which
will have the same cardinality as if the final $t$ was missing. 
Now let
$Z$ be the set of all $a$ in (our final conjugate of) $A$ 
along with the identity.
We can regard $tZ$ as $S$ and $E$ as $Y$ in Proposition 2.11, which
tells us that either there is $a\in A$ (or equal to the identity)
with $|EtaE|=|Y|^2$ thus $|A^3|\ge |YaY|\ge (1/36)^2|A|^2$,
or the subgroup generated by $E$ and $tZ$ is 
infinite cyclic or infinite dihedral. 
But as $t\in tZ$, this is the same as the subgroup
generated by $t,E$ and $A$, which in turn contains $\langle A\rangle$.
\end{proof}

Therefore we have a full understanding of growth in a free product
if we understand growth of subsets in the factors: for instance:
\begin{co}
Let $\Gamma=G_1*\ldots *G_n$ be a free product of groups where the factor
groups $G_i$ are all virtually cyclic, or all virtually abelian, or all
virtually nilpotent. Then for any finite subset $A$ of $\Gamma$, we have
that $|A^3|\ge (1/7776)|A|^2$ unless the subgroup $\langle A\rangle$
is virtually cyclic, respectively virtually abelian, respectively virtually
nilpotent.
\end{co}

A particular case of this is the free product of finite groups and $\Z$:
\begin{co}
Let
$\Gamma=G_1*\ldots *G_n$ be any free product 
of groups where each $G_i$ is either finite or equal to $\Z$.
Then for any finite subset $A$ of $\Gamma$, we have
that $|A^3|\ge (1/7776)|A|^2$ unless the subgroup $\langle A\rangle$
is finite or infinite cyclic, or (when one of the groups $G_i$ has even
order) is equal to the infinite dihedral group $C(2)*C(2)$.
\end{co}

In particular this gives that $|A^3|\ge (1/7776)|A|^2$ for any
finite subset $A$ of a Fuchsian group $F$
(a non elementary discrete subgroup of $PSL(2,\R)$) which is not cocompact,
because any finitely generated subgroup of $F$ is a free product of (finite
or infinite) cyclic groups. Thus this applies to the modular group
$PSL(2,\Z)=C(2)*C(3)$, with $A$ satisfying our growth condition unless
$\langle A\rangle$ is cyclic or equal to $C(2)*	C(2)$. This growth
estimate improves
a result in \cite{raz} which states that for this group we have an
unspecified $d>0$ such that $|A^3|\ge |A|^2/(\mbox{log}|A|)^d$.

\section{Other groups with Helfgott type growth}

First note that the infinite subgroups of a group $G$
with Helfgott type growth 
also have this property, with the same $c$ and $\delta$. 
Moreover because any finite subset $A$ of $G$ will
lie in the finitely generated subgroup $\langle A\rangle$ of $G$, we can
assume $G$ itself is finitely generated, at least at the expense of
replacing an infinitely generated $G$ with the collection of its
finitely generated subgroups.
 
Historically the first group (or infinite sequence of groups) that was thought
of as being most ``free-like'', after free groups themselves and 
free products, was probably the surface group $S_g$, which is the
fundamental group of the closed orientable surface of genus $g\ge 2$. One
would surely hope that this group also demonstrates Helfgott type growth, but
it is not a free product (for instance see \cite{ls} Chapter II Proposition
5.14). It is both an amalgamated free product and an HNN extension, but we
will see in the next section that in general neither of these constructions 
give rise to groups with Helfgott type growth.

However the proof that $S_g$ has Helfgott type growth, with the same
growth expression as in Section 2, follows once we  expand
our interest to a wider class of groups. In fact the proof is surprisingly 
easy provided the right choice of groups is made.

\begin{defn} A group $G$ is {\bf fully residually free} if for any
list of $k$ distinct elements $g_1,g_2,\ldots ,g_k\in G$, we have
a surjective homomorphism $\theta$ from $G$ to a free group $F$ such that
the images $\theta(g_1),\theta(g_2),\ldots ,\theta(g_k)$ are distinct 
elements of $F$.
\end{defn}

Such a group will be torsion free and subgroups of fully residually free
groups are also fully residually free.
Note that the free group $F$ can depend on the choice of elements
in $G$. This property implies that of being residually free but is stronger
in general: for instance
the direct product $G=F_k\times\Z$ is residually
free, but any homomorphism from $G$ to a non abelian free group would
send the generator $z$ of $\Z$ to the identity. Thus if $k\ge 2$ and   
$x,y$ are non commuting elements of $F_k$ then 
on taking the identity, $z$ and the commutator $[x,y]$, we cannot satisfy the
above definition. In fact B.\,Baumslag shows in \cite{ba67} that if a group
$G$ is finitely generated then it is fully residually free if and only if
it is residually free and does not contain $F_2\times\Z$ as a subgroup.

Finitely generated fully residually free groups are also known as 
{\bf limit groups} 
and are important in a number of areas, for instance logic
and topology. Indeed recent results indicate that limit groups have
a very strong claim to be the smallest naturally defined class of torsion
free groups properly containing the free groups $F_k$.
  
The following result is now almost immediate.
\begin{co}
Let $G$ be a fully residually free group and let $A$ be any finite subset
of $G$. Then either 
$|A^3|\ge (1/7776)|A|^2$ or $\langle A\rangle$ is a free abelian subgroup
of $G$.
\end{co}
\begin{proof}
We assume that $\langle A\rangle$ is non abelian, because a finitely
generated abelian subgroup of a torsion free group is free abelian. 
List the elements of $A$ as $\{a_1,a_2,\ldots ,a_n\}$ and assume that
$a_1$ and $a_2$ do not commute. Now find a surjective homomorphism
$\theta:G\mapsto F$ where the images of
the set $A\cup \{[a_1,a_2],id\}$ are distinct. 
As $\theta([a_1,a_2])$ is not the identity, 
we set $S=\theta(A)$ with $|S|=n$
and note that $\langle S\rangle$ is non abelian
(and free as it is a subgroup of the free group $F$). 
Next we apply Corollary 2.12 to $S$ in $F$ (or in $\langle S\rangle$ if $F$
is infinitely generated) to obtain $|S^3|\ge (1/7776)n^2$. But the
triple product $S^3$ is equal to $\theta(A^3)$ so the triple product $A^3$
must be at least as big.
\end{proof}
\begin{co} If $S_g$ is the fundamental group of a closed orientable surface
of genus $g\ge 2$ and $A$ is any finite subset of $S_g$ then either
$|A^3|\ge (1/7776)|A|^2$ or $\langle A\rangle$ is infinite cyclic.
\end{co}
\begin{proof}
By \cite{besthn} Section 1.6, $S_g$ is a limit group.
\end{proof}

In fact the proof of Corollary 3.2 will work (giving exactly the same
conclusion other than the exceptional case, where $\langle A\rangle$ 
could be abelian or infinite dihedral)
for any group which is
fully residually (free product of abelian groups).

We finish this section
by mentioning other groups shown in the literature to have
Helfgott type growth. The only example up till now with $\delta=1$ was
Safin's result on free groups, but previously $\delta>0$ had been
established in some cases. In \cite{che}, which follows closely Helfgott's
original method, Helfgott type growth was established for $SL(2,\C)$
with unspecified $c,\delta>0$ unless $\langle A\rangle$ is finite or
metabelian (note not virtually abelian as claimed there: see the example
after Theorem 4.3). This was apparently
generalised in \cite{vu} Theorem 4.2 which
replaces $\C$ by any characteristic zero integral domain $D$
and which gave the exceptions as $\langle A\rangle$ is finite or metabelian.
However this is actually the same result because if we embed $D$ into
its field of fractions $\mathbb F$, we have that a finitely generated
subgroup of $SL(2,\mathbb F)$ is a subgroup of $SL(2,\C)$, by embedding
$\Q(x_1,\ldots ,x_n)$ into $\C$, where $x_1,\ldots ,x_n$ are the matrix
entries of a generating set. (This is the same argument as used in the
proof of the Tits alternative and in the residual finiteness of finitely
generated linear groups.) Now this need not be true if our group $G$
is infinitely generated but $c,\delta>0$ are absolute constants, and so
will apply to any $A$ by working in $\langle A\rangle$. 
In fact $c$ and $\delta$ are not given explicitly in either paper, which
use quite different proofs.
(Actually all limit groups embed into $SL(2,\R)$ (see \cite{wil} Lemma 2.2)
and so the content of Corollary 3.2 is not that we have Helfgott type growth
but that we can take $\delta=1$.)

Then in \cite{raz} it was shown that for the free group $F_2$
there is (an unspecified)
$d>0$ such that $|A^3|\ge |A|^2/(\mbox{log}|A|)^d$ unless $\langle A\rangle$
is infinite cyclic (or $|A|=1$).
Thus although this does not give $\delta=1$ for growth
in $F_2$, it does so for every $\delta<1$. This is also shown to hold for
virtually free groups $G$ (with infinite cyclic replaced by virtually
cyclic) but now $d$ will depend on $G$, as we will see in the next section.

There are two further basic constructions, both involving finite normal
subgroups, which allow us to obtain new groups with Helfgott type growth
from old ones.
If $G$ has Helfgott type
growth then of course arbitrary
quotients of $G$ need not possess
this property, as every group is a quotient of a free group. However
we do have:
\begin{prop}If $N$ is a finite normal subgroup of the infinite group
$G$ such that
$G/N$ has Helfgott type growth then
$G$ has Helfgott growth for the same $\delta$ but
$c$ replaced by $c/(|N|^{1+\delta})$. Conversely if $G$ has Helfgott
type growth then so does $G/N$, again with the same $\delta$ but
now with $c$ replaced by $c|N|^\delta$. 
\end{prop}
\begin{proof} This is a bit like Corollary 3.2. On being given a finite
subset $A$ of $G$ with $H=\langle A\rangle$ not virtually nilpotent and
taking the image 
$\pi(A)=AN/N$ of $A$ under the natural homomorphism from $G$ to $G/N$,
we have that $\pi(H)=\langle\pi(A)\rangle$ 
is isomorphic to $H/(H\cap N)$ and so is also not virtually
nilpotent. (This need not be true if $H$ were infinitely generated, but
here we can use the fact that $H$ is residually finite because it is
virtually polycyclic.)
Thus $|A^3|\ge |\pi(A^3)|\ge c|\pi(A)|^{1+\delta}$ but
$|\pi(A)|\ge |A|/|N|$.

As for the other way round, on taking $S\subseteq G/N$ with 
$\langle S\rangle$ not virtually nilpotent, let $A\subseteq G$ map
injectively onto $S$ under $\pi$ and consider the subset $AN$ of $G$
with $|AN|=|A||N|$. Certainly $\pi(AN)=S$ so $\langle AN\rangle$ cannot
be virtually nilpotent. Therefore $|(AN)^3|\ge c|AN|^{1+\delta}$ but
$|A^3N|=|S^3||N|$ by counting preimages,
and normality of $N$ tells us that $(AN)^3=A^3N$  
\end{proof}

One curious consequence in Proposition 3.4
about the passage of Helfgott growth from $G$ to $G/N$ is that if
$G$ is isomorphic to $G/N$, so $G$ possesses a surjective
endomorphism with finite but non trivial kernel and implying that
$G$ is non Hopfian, then $G$ cannot have Helfgott type growth for any
$c,\delta>0$ because repeated application of this Proposition would
mean that $c$ was unbounded.   

This allows us to give the best possible $\delta$ and a specific
$c>0$ for $SL(2,\Z)$, improving the results in
\cite{che} and \cite{raz} for this group.
\begin{co}
The group $SL(2,\Z)$ has Helfgott type growth (unless $\langle A\rangle$ is
virtually cyclic) with $\delta=1$ and $c=1/31104$.
\end{co}
\begin{proof} The quotient $PSL(2,\Z)$ of $SL(2,\Z)=C_4*_{C_2}C_6$
satisfies Corollary 2.12 and the kernel has size 2.
\end{proof} 

In fact this argument works whenever we have an amalgamated free product
$G=A*_NB$ of groups $A$ and $B$ where the amalgamated subgroup $N$ is
finite and normal in both $A$ and $B$. This is because $N$ is then normal
in $G$ (as $A\cup B$ generates $G$) and $G/N$ is then isomorphic to the
free product $A/N*B/N$, so Corollary 2.12 applies here.

Moreover if we want further examples with some $\delta>0$ then we can take
a finite collection of groups known to have Helfgott type growth along with
the minimum $\delta$ for these groups, then form their free product $\Gamma$
which will 
not reduce $\delta$ by Corollary 2.12. We can then do other things, for
instance any $G$ with
$G/N=\Gamma$ for $N$ a finite normal subgroup
will also have Helfgott type growth with the
same $\delta$, and then we can form more free products and so on.

\section{Groups without Helfgott type growth}
 
We first show that $\delta=1$ is the best possible value for a group
with Helfgott type growth, at least away from one particular class of groups.

\begin{prop}
Suppose $G$ is a finitely generated
infinite group which is not of bounded exponent (meaning
there are either elements of infinite order or no upper bound to the orders
of elements in $G$) and which is not virtually nilpotent. 
If there is $c,\delta>0$ such that
all finite subsets $A$ of $G$ with $\langle A\rangle$ not virtually nilpotent
satisfy
$|A^3|\ge c|A|^{1+\delta}$ then $\delta\le 1$.
\end{prop}
\begin{proof}
Let $G=\langle g_1,\ldots ,g_l\rangle$ and let $x_N$ be an element
of $G$ with infinite order, or order at least $2N$. We take
$A=\{g_1,\ldots ,g_l,x_N,x_N^2,\ldots , x_N^N\}$. Thus
$N\le|A|\le N+l$ with $\langle A\rangle=G$ but a quick count reveals that
$|A^2|\le 2(l+1)N-2+l^2$. This means that $|A^3|\le |A||A^2|<2(l+1)N(N+l)
+l^2(N+l)$ so Helfgott type growth would imply that the right hand side of this
inequality is always greater than $cN^{1+\delta}$. 
We now let $N\rightarrow\infty$ to get a contradiction if $\delta>1$.
\end{proof}

Note that the existence of finitely generated infinite groups with
bounded exponent is highly non trivial. We have a result on this in
\cite{hru} using model theory: Corollary 4.18 states that there is
a function $f(K,e)$ such that if $G$ is any group with exponent
dividing $e$ and $A$ is a $K$-approximate group in $G$ then there
exists a subgroup $H$ of $G$ such that $A$ is contained in $f(K,e)$ left
cosets of $H$. In particular let us take the examples 
of Ol'shanski\u{i} where
for a sufficiently large prime $p$ we have a 2 generator infinite group
$G(p)$
of exponent $p$ such that the only proper non trivial subgroups are
cyclic of order $p$. Then this result implies that any $K$-approximate
group $A$ in $G(p)$ has $|A|\le pf(K,p)$ so it could happen that
$G(p)$ has Helfgott type growth if $f$ were polynomial in $K$. However    
this function is not given explicitly.

Let us stay with the basic example in Proposition 4.1 and use it to
construct groups which do not have Helfgott type growth for any
$\delta>0$.
We see that $A^3$ is made up of eight
different types of product, according to whether we choose an element
$g_i$ or a power
$x_N^i$ in each of the three places. It is also clear that the size of
seven of these eight types are each linear in $N$ and a total bound for
these seven is $l^3+3l^2N+2l(2N-1)+3N-2$. 
But the interesting point is the size of
$\{x_N^ig_kx_N^j\}$ for $1\le i,j\le N$ and $1\le k\le l$,
which could be as high as $lN^2$. In particular there will be many
groups where we can effectively take $l=1$ above, by
finding elements $x$ of infinite order and $g$
such that $\langle g,x\rangle$ is not virtually nilpotent and
$\{x^igx^j:1\le i,j\le N\}$ has size $N^2$. Thus we
obtain sets $A$ with $\langle A\rangle$ not virtually nilpotent
and $|A^2|<4|A|$ but $|A^3|/|A|\rightarrow\infty$ as $N\rightarrow\infty$,
as mentioned in the introduction.

We also mentioned in the introduction the case of $F_n\times\Z$ which
is the only example without Helfgott type growth given so far. It might
seem that this is a particular manifestation of the direct product but
in fact it is much more general.
\begin{prop}
If $G$ is a group possessing an element $z$ of infinite order such that
its centraliser $C(z)$ in $G$ contains a finitely generated subgroup
which is not virtually nilpotent then $G$ does not have Helfgott type
growth.
\end{prop}
\begin{proof}
If $H$ is such a subgroup (which contains $z$ without loss of generality)
then extend $\{z\}$ to a finite generating set 
$\{z,g_1,\ldots ,g_l\}$ of $H$ where no $g_i$ is in $\langle z\rangle$. Thus
the set $A_N=\{z,\ldots ,z^N,g_1,\ldots ,g_l\}$ of size $N+l$ has
$\langle A_N\rangle=H$ and any element in $A_N^3$ of the form
$z^ig_kz^j$ can be written as $g_kz^{i+j}$, thus $|A_N^3|/|A_N|$ is bounded. 
\end{proof}

Another variation on Proposition 4.1 is to take an element $x$ of
infinite order in $G$ along with $y\notin\langle x\rangle$ and
$A_d=\{y,x,x^2,\ldots ,x^d\}$. Then as
pointed out above, a necessary condition to force $|A_d^3|/|A_d|$
to be bounded is that a collision occurs, namely there exist
$1\le i,j,k,l\le d$ such that $x^iyx^j=x^kyx^l$, since otherwise
$|A_d^3|\ge d^2$. Now $i=k$ gives
rise to a contradiction on the order of $x$ unless $j=l$ as well, so
no collision has occurred. Similarly $j\neq l$ and so we have the
relation $yx^{j-l}y^{-1}=x^{k-i}$ holding in our group $G$.
This suggests the famous
family in combinatorial group theory of Baumslag-Solitar groups
$BS(m,n)$ for integers $m,n\ne 0$ which are defined by the presentation
$\langle x,y|yx^my^{-1}=x^n\rangle$, giving us:

\begin{thm} If a group $G$ contains a Baumslag-Solitar group
$BS(m,n)=\langle x,y|yx^my^{-1}=x^n\rangle$ where $|m|,|n|$ are not
both equal to 1 (or
even the image of a Baumslag-Solitar group where
$x$ has infinite order and such that the image is not virtually nilpotent) 
then $G$ does not have Helfgott type growth. 
\end{thm}
\begin{proof}
On taking the subset $A_d=\{y,x,x^2,\ldots ,x^d\}$ with size $d+1$
obtained from
this Baumslag-Solitar subgroup $\langle x,y\rangle$ (or image thereof),
we will show that $|A_d^3|<(10+|m|+|n|)|A_d|$. Therefore if there
was $c,\delta>0$ with $|A_d^3|\ge c|A_d|^{1+\delta}$, it would imply
that $|A_d|$ is bounded as $d\rightarrow\infty$.

Let us consider the elements in $A_d^3$ of the form $x^iyx^j$ where
$1\le i,j\le d$. We can assume without loss of generality that
$m>0$ by replacing $x$ with $x^{-1}$. We now 
set $j=qm+r$ where $0\le r<m$, so that $0\le q\le d/m$.
Note that the
relation $yx^my^{-1}=x^n$ implies $yx^{qm}=x^{qn}y$.
Thus $x^iyx^{qm}x^r=x^{qn+i}yx^r$ with $1\le qn+i\le d(1+n/m)$ when
$n>0$ and $1+d(n/m)\le qn+i\le d$ when $n<0$. Thus we have at
most $(1+|n|/m)d$ possibilities for the first power of $x$ and
$m$ for the second, telling us that $\{x^iyx^j\}$ has size less than
$(m+|n|)|A_d|$.
\end{proof}

The reason for excluding the group $BS(\pm 1,\pm 1)$ is because it is
either $\Z\times\Z$ or has
$\Z\times\Z$ as an index 2 subgroup, so is virtually nilpotent.
Otherwise $BS(m,n)$ is not virtually nilpotent, indeed
if neither $|m|$ nor $|n|$ is equal to 1 then
$BS(m,n)$ contains a non abelian free group (although it does satisfy
the conditions of Proposition 4.2).
The interesting case is when $|m|=1$, say $m=1$ without loss of generality,
but $|n|\ne 1$. Here $BS(1,n)$ is metabelian but not virtually nilpotent.
Also the centraliser of any element in $BS(1,n)$ is abelian (though it
need not be finitely generated abelian) so these are genuinely new examples
not covered by Proposition 4.2.
Moreover $BS(1,n)$ is linear, in fact is a subgroup
of $SL(2,\C)$ on taking
\[x=\left( \begin{array}{rr}
1 & 1 \\
0 & 1 
\end{array} \right) \qquad\mbox{ and }\qquad
y= \left( \begin{array}{rr}
\sqrt{n} & 0 \\
0 & \sqrt{1/n}
\end{array} \right)
\]
so is also a subgroup of $SL(2,\R)$ for $n\ge 2$ and is even in
$SL(2,\Q)$ when $n$ is a square. We mention this because in \cite{che}
it is claimed that $SL(2,\C)$ has Helfgott type growth 
with the exception of virtually
abelian subgroups, but on putting $n=4$ in these matrices
we have that the set $A_d$ in $BS(1,4)$ from Theorem 4.3
satisfies $|A_d|^3<15|A_d|$ even though $\langle A_d\rangle$ is not
virtually abelian. The correct statement is that there is Helfgott type
growth away from metabelian and finite subgroups, as shown in Theorem
4.2 of \cite{vu}.

Baumslag Solitar groups are fundamental examples of HNN extensions in that
the base and amalgamated subgroups are just copies of $\Z$. We can do the
same to form an amalgamated free product, such as the group
$\langle x,y|x^2=y^3\rangle$ which is not virtually nilpotent; indeed it
is the fundamental group of the trefoil
knot. Now the infinite order element $z=x^2=y^3$ is central (as it
commutes with both $x$ and $y$), so this group does not have Helfgott
type growth by Proposition 4.2.

One situation which does not seem so neat is when 
an infinite group $G$ has a finite
index (and normal without loss of generality) 
subgroup $H$ with Helfgott type growth. 
On taking a finite subset $A$ in $G$, we certainly
have a coset $Hg$ for $g\in G$ with $|B|\ge |A|/i$, where 
$B=A\cap Hg$ and $i$ is the index of $H$ in $G$. 
We then try to estimate
$|B^3|$ but this is equal to $|XYZ|$ where $X,Y,Z\subseteq H$ with    
$X=Bg^{-1}$, $Y=gXg^{-1}$ and $Z=g^2Xg^{-2}$. Although $|X|=|Y|=|Z|$,
we now need results on the size of triple products of different subsets
in $H$ for this approach to work. (There is also the problem that
$\langle A\rangle$ not being virtually nilpotent does not imply the same for
$\langle B\rangle$.)
Arguments along these lines were given in \cite{raz} for the free
group of rank 2 and so a result was obtained for virtually free groups:
if $G$ is a virtually free group then there is $d>0$ such
that for any finite subset $A$ of $G$, either $\langle A\rangle$ is
virtually cyclic or $|A^3|\ge |A|^2/(\mbox{log}|A|)^d$. Thus although this
does not quite establish Helfgott type
growth with $\delta=1$ for virtually
free groups, it does so for any $\delta<1$.

However consider virtually free groups of the form
$G=F_2\times N$ when $N$ is
a finite group. On taking $A=\{(x,n),(y,n):n\in N\}$ for 
$F_2=\langle x,y \rangle$,
we have that $\langle A\rangle$ contains a copy of $F_2$
with $|A|=2|N|$ but $|A^3|=8|N|$. Thus although $G$ has Helfgott type
growth
with $\delta=1$ by Proposition 3.4, we see by increasing $N$ that there
is no $c,\delta>0$ (nor $d$ in the above expression) that will work for
all virtually free groups. 

We finish with some questions which are an attempt to back up the
author's claim in the introduction about the lack of examples in the
area.\\
\hfill\\
1. Is there a group which has Helfgott type growth for some $\delta>0$
but which does not when $\delta=1$?\\
\hfill\\
2. Following on from Proposition 4.2, is there a group where the
centraliser of every element of infinite order is finitely generated and
virtually nilpotent (and with elements of infinite order)
but which does not have Helfgott type growth? Or even an example proven
not to have Helfgott type growth for $\delta=1$? How about $SL(3,\Z)$?\\
\hfill\\
3. In word hyperbolic groups, the centraliser of an infinite order element
is always virtually cyclic so the condition in Question 2 is satisfied.
(They also do not contain Baumslag-Solitar groups.) 
Therefore we can ask whether all word hyperbolic groups 
have Helfgott type growth and, if so, what values of $\delta$
are obtained. If a word hyperbolic group is a free product then its free
factors are word hyperbolic too, so it is enough to consider the freely
indecomposable case by Corollary 2.12.\\
\hfill\\
4. Do we have Helfgott type growth with $\delta=1$ (but necessarily
varying $c$) for all virtually free groups or for all virtually
surface groups? In the latter case, what about specialising to triangle
groups?

\end{document}